\documentclass{birkjour}
\usepackage{amsfonts,amssymb,amsmath,amsgen,amsopn,amsbsy,graphicx,epsfig,bm}
\usepackage{amscd,bezier,latexsym,mathrsfs,enumerate,multirow,cite}

 \newtheorem{thm}{Theorem}[section]
 \newtheorem{cor}[thm]{Corollary}
 
 \newtheorem{prop}[thm]{Proposition}
 \theoremstyle{definition}
 \newtheorem{defn}[thm]{Definition}
 \theoremstyle{remark}
 \newtheorem{rem}[thm]{Remark}
 
 \numberwithin{equation}{section}

\begin{document}

%
%
%
%
%
%
%
%
%

\title[A Revisit to $n$-Normed Spaces through Its Quotient Spaces]
 {A Revisit to $n$-Normed Spaces through \\ Its Quotient Spaces}

\author[H. Batkunde]{Harmanus Batkunde$^*$}

\address{%
Analysis and Geometry Research Group,\\
Faculty of Mathematics and Natural Sciences,\\
Bandung Institute of Technology\\
Bandung 40132, Indonesia}

\email{(correspondence) h.batkunde@fmipa.unpatti.ac.id}

\author[H. Gunawan]{Hendra Gunawan}
\address{%
	Analysis and Geometry Research Group,\\
	Faculty of Mathematics and Natural Sciences,\\
	Bandung Institute of Technology\\
	Bandung 40132, Indonesia}
\email{hgunawan@math.itb.ac.id}
\subjclass{Primary 46B20; Secondary 54B15}

\keywords{continuous mappings, fixed points, $n$-normed spaces, quotient spaces.}

\date{March 29, 2019}

\begin{abstract}
In this paper, we study some features of $n$-normed spaces with respect to
norms of its quotient spaces. We define continuous functions with respect to the
norms of its quotient spaces and show that all types of continuity are equivalent.
We also study contractive mappings on $n$-normed spaces using the same approach.
In particular, we prove a fixed point theorem for contractive mappings on
a closed and bounded set in an $n$-normed space.
\end{abstract}

\maketitle
\section{Introduction}
We will study some features of the $n$-normed spaces with its quotient spaces as a tool.
In the 1960's, the concept of $n$-normed spaces was investigated by S. G\"ahler \cite{gah1,gah2,gah3,gah4}.
Let $n$ be a nonnegative integer and $X$ be a real vector space with $\dim(X) \geq n$.
The pair $\left(X,\|\cdot,\dots,\cdot\|\right)$ is called an $\bm{n}$\textbf{-normed space}
where $\|\cdot,\dots,\cdot\| : X^{n} \to \mathbb{R}$  is an \textbf{$\bm{n}$-norm} on $X$  which
satisfies the following conditions:
\begin{itemize}
	\item[(i)] $\|x_1,\dots,x_n\| \geq 0$; $\|x_1,\dots,x_n\| = 0$ if and only if $x_1,\dots,x_n$ linearly dependent.
	\item[(ii)] $\|x_1,\dots,x_n\|$ is invariant under permutation.
	\item[(iii)] $\|\alpha x_{1},\dots,x_{k}\|=|\alpha| \ \|x_{1},\dots,
	x_{k}\|$ for any $\alpha \in \mathbb{R}$.
	\item[(iv)] $\|x_1+x_1^\prime,x_{2},\dots,x_{k}\| \leq \|x_1,x_{2},
	\dots,x_{k}\| + \|x_1^\prime,x_{2},\dots,x_{k}\|$.
\end{itemize}

Some researchers studied further various aspects of these spaces \cite{bgp,br,gm,gsn,ht,ert,mei,SS}.
In particular, Brzd\c{e}k and Ciepli\'{n}ski \cite{bc} proved a fixed point theorem for operators acting
on some classes of functions, with values in $n$-Banach spaces.
Ekariani \textit{et al.} \cite{eka} also proved a fixed point theorem of a contractive mapping
in $\ell^{p}$ as an $n$-normed space. They used a norm which is equivalent to the usual norm on $\ell^{p}$.
The norm that they used is simpler than the normed used by Brzd\c{e}k and Ciepli\'{n}ski.
This norm is derived from the $n$-norm using a linearly independent set consisting of $n$ vectors.

In this paper, we provide a more general approach to prove a fixed point theorem in an $n$-normed space.
First, we define some quotient spaces of an $n$-normed space. In each quotient space we define a norm
which is derived from the $n$-norm in a certain way. This norm is simpler than the norm that Brzd\c{e}k
and Ciepli\'{n}ski or Ekariani \textit{et al.} used. Using these norms, we investigate continuous mappings
and contractive mappings in the $n$-normed space and prove a fixed point theorem for contractive
mappings on a closed and bounded set in the $n$-normed space. Using this approach we will also prove a fixed point
theorem in $\ell^{p}$ as an $n$-normed space. Compared to the approaches used by previous researchers \cite{gsn,ht,mei,bc,eka},
our approach provides a more general way which can be used to study properties of $n$-normed spaces.

\section{Preliminaries}
\label{Sec:2}

Let us begin with the construction of quotient spaces of an $n$-normed space.
Let $(X,\|\cdot,\dots,\cdot\|)$ be an $n$-normed space, $m\in \{1,\dots,n\}$ and $Y=\{y_{1},\dots,y_{n}\}$ be a
linearly independent set in $X$. For $\{i_{1},\dots,i_{m}\} \subset \{1,\dots,n\}$, consider $Y\setminus \{y_{i_{1}},\dots, y_{i_{m}}\}$.
We define the following subspace of $X$:
\[
Y_{i_{1},\dots, i_{m}}^{0}:= {\rm span} \, Y \setminus \{y_{i_{1}},\dots, y_{i_{m}}\} :=
\left\{\sum_{i\notin \{i_{1},\dots,i_{m}\}} \alpha_{i} y_{i} \, \, ;  \, \,\alpha_{i} \in
\mathbb{R}.\right\}.
\]
For each $u \in X$, the corresponding coset in $X$ is defined by
\[
\overline{u}:=\left\{u+\sum_{i \notin \{i_{1},\dots,i_{n}\}} \alpha_{i} y_{i}: \alpha_{i} \in \mathbb{R}\right\}.
\]
Then we have if $\overline{u} = \overline {v}$, then $u-v \in {\rm span} \, Y \setminus \{y_{i_{1}},\dots, y_{i_{m}}\}$.
We define the quotient space
\[ \label{e04}
X^{*}_{i_{1},\dots, i_{m}}:=X/Y_{i_{1},\dots, i_{m}}^{0}:=\{\overline{u}: u \in X\}.
\]
The addition and scalar multiplication apply in this space.
Moreover, we define the following norm on $X^{*}_{i_{1},\dots, i_{m}}$:
\begin{equation} \label{e05}
\begin{aligned}
\|\overline{u}\|^{*}_{i_{1},\dots, i_{m}}:=& \|u,y_{1},\dots,y_{i_{1}-1},y_{i_{1}+1},\dots,y_{n}\|+\dots \\
&+ \|u,y_{1},\dots,y_{i_{m}-1},y_{i_{m}+1},\dots,y_{n}\|.
\end{aligned}
\end{equation}
Using the above construction, we get $\binom{n}{m}$ quotient spaces. Moreover, for an $m \in \{1,\dots,n\}$,
the set that contains all quotient spaces constructed above is called \textbf{class-$\bm{m}$ collection}
\cite{bgn}.

Observe that each term of (\ref{e05}) is a norm of quotient spaces in class-$1$ collection.
Here (\ref{e05}) can be written as
$$
\|\overline{u}\|^{*}_{i_{1},\dots, i_{m}}=\|\overline{u}\|^{*}_{i_{1}}+\dots + \|\overline{u}\|^{*}_{i_{m}}.
$$
One can see that the norm we define in (\ref{e05}) is more simpler than the norm that Brzd\c{e}k
and Ciepli\'{n}ski used since we define it with respect to a linearly independent set consisting of $n$ vectors.

For an $m \in \{1,\dots,n\}$ we will use the phrase 'norms of class-$m$ collection' which means
'all norms of each quotient space in the class-$m$ collection'. Furthermore, using the norms of class-$m$
collection, we have observed some topology characteristics of an $n$-normed space as
presented in the following.

\begin{defn}
	\rm \cite{hh2} Let $(X,\|\cdot,\dots,\cdot\|)$ be an $n$-normed space and $m \in \{1,\dots,n\}$.
	We say a sequence $\{x_{k}\} \subset X$ \textbf{converges with respect to the norms of class-$\bm{m}$
		collection to $x$} if for any
	$\epsilon >0 $, there exists an $N\in \mathbb{N}$ such that for $k \geq N$ we have
	\[
	\|\overline{x_{k}-x}\|^{*}_{i_{1},\dots,i_{m}} < \epsilon,
	\]
	for every $\{i_{1},\dots,i_{m}\} \subset \{1,\dots,n\}$.
	In this case we also say
	\[\lim_{n \to \infty} \|\overline{x_{k}-x}\|^{*}_{i_{1},\dots,i_{m}} = 0,\]
	for every $\{i_{1},\dots,i_{m}\} \subset \{1,\dots,n\}$.
	If $\{x_{k}\}$ does not converge, we say it \textbf{diverges}.
\end{defn}

By the above definition, we have the following theorem.

\begin{thm}\label{convT}
	{\rm \cite{hh2}} Let  $(X,\|\cdot, \dots, \cdot\|)$ be an $n$-normed space and $m \in\{1, \dots,n\}$.
	A sequence $\{x_{k}\} \subset X$ is convergent with respect to the norms of class-$1$ collection
	if and only if it is convergent with respect to the norms of class-$m$ collection.
\end{thm}

\begin{cor}
	Let $(X,\|\cdot, \dots, \cdot\|)$ be an $n$-normed space and $m_{1},m_{2} \in \{1,\dots,n\}$.
	Then, if a sequence in $X$ converges with respect to the norms class-$m_{1}$ collection, then it
	is also converges with respect to the norms of class-$m_{2}$ collection.
\end{cor}

\begin{defn}
	{\rm \cite{hh2}} Let $(X,\|\cdot,\dots,\cdot\|)$ be an $n$-normed space and $m \in \{1,\dots,n\}$. 
	A sequence $\{x_{k}\} \subset X$ is called a \textbf{Cauchy sequence with respect to the norms of class-$\bm{m}$ 
		collection} if for any $\epsilon >0$, there exists an $N\in \mathbb{N}$ such that, for every $k,l \geq N$, 
	we have
	\[
	\|\overline{x_{k}-x_{l}}\|^{*}_{i_{1},\dots,i_{m}}< \epsilon,
	\]
	for every  $\{i_{1},\dots,i_{m}\} \subset \{1,\dots,n\}$. In other words
	\[
	\lim_{k,l \to \infty} \|\overline{x_{k}-x_{l}}\|^{*}_{i_{1},\dots,i_{m}}=0,
	\]
	for every  $\{i_{1},\dots,i_{m}\} \subset \{1,\dots,n\}$.
\end{defn}

\begin{thm}\label{cauT} {\rm \cite{hh2}}
	Let $(X,\|\cdot,\dots,\cdot\|)$ be an $n$-normed space and $m \in \{1,\dots,n\} $.
	If $\{x_{k}\}$ is convergent with respect to the norms of class-$1$ collection,
	then $\{x_{k}\}$ is Cauchy with respect to the norms of class-$m$ collection.
\end{thm}

\begin{cor}\label{ceq}
	Let $(X,\|\cdot, \dots, \cdot\|)$ be an $n$-normed space and $m_{1},m_{2} \in \{1,\dots,n\}$.
	Then, if a sequence in $X$ is a Cauchy sequence with respect to the norms class-$m_{1}$ collection, then it
	is also a Cauchy sequence with respect to the norms of class-$m_{2}$ collection.
\end{cor}

\begin{rem}
	For $m_{1},m_{2} \in \{1,\dots,n\}$, we find that all types of convergent sequence with respect to the norms of class-$m_{1}$
	collection and to the norms of class-$m_{2}$ colection are equivalent. The equivalence also applies to all
	types of Cauchy sequence with respect to the norms of class-$m_{1}$ collection and to the norms of class-$m_{2}$ collection.
	In this regard, we may simply use the word 'converges or Cauchy' instead of 'converges or Cauchy with respect to the norms of
	class-$m$ collection'. Furthermore if every Cauchy sequence in $X$ converges, then $X$ is \textbf{complete}.
	By the word 'complete', we mean 'complete with respect to the norms of class-$m$ collection', for some $m \in \{1,\dots,n\}$.
\end{rem}

Now let us move to the definition of closed sets with respect to the norms of class-$m$ collection,
for any  $m \in \{1,\dots,n\}$.

\begin{defn}
	{ \rm \cite{hh2}} Let $(X,\|\cdot,\dots,\cdot\|)$ be an $n$-normed space and $K \subseteq X$.
	The set $K$ is called \textbf{closed} if for any sequence $\{x_{k}\}$ in $K$ that converges
	in $X$, its limit belongs to $K$.
\end{defn}

Note that by saying 'closed' we mean 'closed with respect to class-$m$ collection', for some $m \in \{1,\dots,n\}$.
Next is the definition of bounded sets.

\begin{defn}
	{\rm \cite{hh2}} Let $(X,\|\cdot, \dots, \cdot\|)$ be an $n$-normed space,  $m \in \{1,\dots,n\}$
	and $K \subseteq X$ be a nonempty set. The set $K$ is called \textbf{bounded with respect to the norms of
		class-$\bm{m}$ collection} if and only if for any $x \in K$ there exists an $M>0$ such that
	\[
	\|\overline{x}\|^{*}_{i_{1},\dots,i_{m}}\leq M,
	\]
	for every $\{i_{1},\dots,i_{m}\} \subset \{1,\dots,n\}$.
\end{defn}

We also have the following theorem which says that all types of bounded set are equivalent.

\begin{thm} \label{t1.5}
	{\rm \cite{hh2}} Let  $(X,\|\cdot, \dots, \cdot\|)$ is an $n$-normed space, $m \in \{1,\dots,n\}$
	and $K \subset X$ nonempty. The set $K$ is bounded with respect
	to the norms of class-$1$ collection if and only if it is bounded with respect to class-$m$ collection.
\end{thm}

Furthermore, for the completeness of the $n$-normed space with respect to the norms of class-$m$ collections
we have the following theorem.

\begin{thm}\label{comp}
	Let $(X,\|\cdot,\dots,\cdot\|)$ be an $n$-normed space and $m \in \{1, \dots,n\}$. Then
	$X$ is complete with respect to the norms of class-$1$ collection if and only if $X$ is
	complete with respect to the norms of class-$m$ collection.
\end{thm}

\begin{proof} Suppose that $X$ is complete with respect to the norms of class-$1$
	collection and $m \in \{1,\dots,n\}$. Take any Cauchy sequence $\{x_{k}\}$ with respect to the
	norms of class-$m$ collection in $X$. Then by Theorems \ref{convT} and \ref{cauT} we have $\{x_{k}\}$
	is a convergent sequence with respect to the norms of class-$m$ collection in $X$. This tells us that
	$X$ is complete with respect to the norms of class-$m$ collection. The converse is similar.
\end{proof}

\begin{rem}
	As in Corollary \ref{ceq}, one can see that the equivalence also applies to boundedness and completeness of
	a set. Then, we will use the word 'bounded' instead of the phrase 'bounded with respect
	to the norms of class-$m$ collection', for some $m\in \{1,\dots,n\}$.
\end{rem}

\begin{rem}
	For a fixed $ m \in \{1, \dots,n\}$, the convergence of a sequence, the closedness and the boundedness of
	a set with respect to the norms of class-$m$ collection may be investigated with respect to some norms
	$\|\cdot\|^{*}_{i_{1},\dots,i_{m}}$ we choose such that
	\[
	\bigcup \{i_{1},\dots,i_{m}\} \supseteq \{ 1,\dots,n\}.
	\]
	Moreover, the least number of norms that can be used to investigate these notions is
	$\left\lceil \frac{n}{m}\right\rceil$. Next, by using a similar approach we will study
	continuous mappings, contractive mappings, and also prove a fixed point theorem of
	contractive mappings of a closed and bounded in an $n$-normed space.
\end{rem}

\section{Continuous Mappings with Respect to the Norms of Class-$m$ Collection}

We shall now discuss the continuity of a mapping with respect to the norms of class-$m$ collection in
an $n$-normed space.

\begin{defn}\label{d1}
	Let $(X,\|\cdot,\dots,\cdot\|^{*})$ be an $n_{1}$-normed space and $(Z,\|\cdot,\dots,\cdot\|^{**})$
	be an $n_{2}$-normed space and $l \in \{1,\dots,n_{1}\}$, $m \in \{1,\dots,n_{2}\}$. Suppose that $f: X \to Z$.
	\begin{enumerate}
		\item[(i)] We say that $f$ is \textbf{continuous with respect to the norms of class-$\bm{(l,m)}$ collections at $\bm{a \in X}$}
		if and only if for any $\epsilon > 0$ there exists a $\delta > 0$ such that for $x \in X$ with
		$\|\overline{x-a}\|^{*}_{i_{1},\dots,i_{l}}< \delta$ for every $\{i_{1},\dots,i_{l}\} \subset
		\{1,\dots,n_{1}\}$, we have $\|\overline{f(x)-f(a)}\|^{**}_{j_{1},\dots,j_{m}} < \epsilon$
		for every $\{j_{1},\dots,j_{m}\} \subset \{1,\dots,n_{2}\}$.
		\item[(ii)] We say that $f$ is \textbf{continuous with respect to the norms of class-$\bm{(l,m)}$ collections on $\bm{X}$}
		if and only if $f$ continuous with respect to the norms of class-$(l,m)$ collection at each $x\in X$.
	\end{enumerate}
\end{defn}

\noindent Moreover, if $Z=X$, then we say '$f$ is continuous with respect to the norms of class-$m$ collections'
instead of 'continuous with respect to the norms of class-$(m,m)$ collections'.
The norms $\|\cdot\|^{*}_{i_{1},\dots,i_{l}}$, $\|\cdot\|^{**}_{j_{1},\dots,j_{m}}$ are
the norms of class-$(l,m)$ collections in $X$ and $Z$ respectively. Note that we define class-$(l,m)$
collections by using a linearly independent set consisting of $n_{1}$ and $n_{2}$ vectors in $X$ and $Z$
respectively. Based on the definition, we have the following theorem.

\begin{thm}\label{c1m}
	Let $(X,\|\cdot,\dots,\cdot\|^{*})$ be an $n_{1}$-normed space and $(Z,\|\cdot,\dots,\cdot\|^{**})$
	be an $n_{2}$-normed space and $l \in \{1,\dots,n_{1}\}$, $m \in \{1,\dots,n_{2}\}$. A mapping
	$f: X \to Z$ is continuous with respect to class-$1$
	collections if and only if $f$ is continuous with respect to class-$(l,m)$ collections.
\end{thm}

\begin{proof} Let $f$ be a continuous mapping with respect to the norms of class-$1$
	collections at $a \in X$, $l \in \{1,\dots,n_{1}\}$ and $m \in \{1,\dots,n_{2}\}$. For any $\epsilon > 0$,
	there is $\delta > 0$ such that for any $x \in X$ with $\|\overline{x-a}\|^{*}_{i} < \frac{\delta}{l}$
	for every $i \in \{1,\dots,n_{1}\}$, we have $\|\overline{f(x)-f(a)}\|^{*}_{j} < \frac{\epsilon}{m}$ for every
	$j \in \{1,\dots,n_{2}\}$. Then for any $x \in X$ with $\|\overline{x-a}\|^{*}_{i_{1},\dots,i_{l}}< \delta$ 
	for every $\{i_{1},\dots,i_{l}\}
	\subset \{1,\dots,n_{1}\}$,	we have $\|\overline{f(x)-f(a)}\|^{*}_{j_{1},\dots,j_{m}} < \epsilon$
	for every $\{j_{1},\dots,j_{m}\} \subset \{1,\dots,n_{2}\}$. This means $T$ is continuous with respect to
	the norms of class-$(l,m)$ collection.
	
	\noindent The converse is obvious. 
\end{proof}

\begin{cor}
	Let $(X,\|\cdot,\dots,\cdot\|^{*})$ and $(Z,\|\cdot,\dots,\cdot\|^{**})$ be an $n_{1}$-normed space
	and an $n_{2}$-normed space respectively, $l_{1}, l_{2} \in \{1, \dots,n_{1}\}$ and $m_{1}, m_{2} \in
	\{1, \dots,n_{2}\}$. If a mapping $T: X\to Z$ is continuous with respect to class-$(l_{1},m_{1})$, then
	$T$ continuous with respect to the norms of class-$(l_{2},m_{2})$.
\end{cor}

\begin{rem}
	By the above corollary, we can see that all types of continuity of a mapping with respect to the norms of
	class-$(l,m)$ collections are equivalent.
	Then from now on we will use the word 'continuous' instead of 'continuous with respect to the norms of
	class-$(l,m)$ collection'.
\end{rem}

Now we present a proposition that gives a relation between a convergent sequence and a continuous
mapping with respect to the norms of class-$(l,m)$ collection.

\begin{prop}\label{p1}
	Let $(X,\|\cdot,\dots,\cdot\|^{*})$ and $(Z,\|\cdot,\dots,\cdot\|^{**})$ be an $n_{1}$-normed space
	and an $n_{2}$-normed space respectively, $l \in \{1,\dots,n_{1}\}$, $m\in \{1,\dots,n_{2}\}$.
	Suppose that $T: X \to Z$ is continuous. If $\{x_{k}\}$ converges to $x$, then
	\[
	\lim _{x\to \infty} Tx_{n}=Tx.
	\]
\end{prop}

\noindent\textit{Proof.} Let $l \in \{1,\dots,n_{1}\}$, $m \in \{1,\dots,n_{2}\}$, and
$\epsilon >0 $. Since $T$ is continuous, there is a $\delta >0$ such that if $\|\overline{y-x}\|_{i_{1},
	\dots,i_{m}}^{*}<\delta$ for every $\{i_{1},\dots,i_{m}\} \subset \{1,\dots,n_{1}\}$, then
$\|Ty-T{x}\|_{j_{1},\dots,j_{l}}^{**} < \epsilon$.
Since $\{x_{k}\}$ converges to $x$ with respect to norms of class-$m$, choose $N \in \mathbb{N}$ such
that for $k \geq N$ we have $\|x_{k}-x\|_{i_{1},\dots,i_{m}} \leq \delta$. It then follows that
$\|\overline{Tx_{k}-Tx}\|_{j_{1},\dots,j_{m}}^{**}<\epsilon$ whenever $k \geq N$. Therefore $Tx_{k}$
converges to $Tx$.

\section{Fixed Point Theorem for Contractive Mappings with Respect to The Norms of Class-$m$ Collection}

In this section, we shall discuss contractive mappings with respect
to the norms of class-$m$ collection in an $n$-normed space and its fixed point theorem.

\begin{defn}
	Let $(X,\|\cdot,\dots,\cdot\|)$ be an $n$-normed space and $m \in \{1,\dots,n\}$.
	A mapping $T: X \to X$ is called \textbf{contractive with respect to the norms of class-$\bm{m}$ collection}
	if there is a $C \in (0,1)$ such that for any $x,y \in X$ we have
	\[
	\|\overline{Tx-Ty}\|^{*}_{i_{1},\dots, i_{m}} \leq C\, \|\overline{x-y}\|^{*}_{i_{1},\dots, i_{m}},
	\]
	for every $\{i_{1},\dots, i_{m}\} \subset \{1,\dots,n\}$ with $i_{1} < \dots < i_{m}$
\end{defn}

Following the definitions of continuous mappings and contractive mappings, we have the following
proposition and theorem.

\begin{prop}\label{p2}
	Let $(X,\|\cdot,\dots,\cdot\|)$ be an $n$-normed space and $m \in \{1,\dots,n\}$.  If $T$ is a
	contractive mapping with respect to the norms of class-$m$ collection, then $T$ is continuous.
\end{prop}

\begin{proof} For an $m \in \{1,\dots,n\}$, let $T$ be a contractive mapping with respect to
	the norms of class-$m$ collection, then there is a $C \in (0,1)$ such that for any $x,y \in X$ we have
	\[
	\|\overline{Tx-Ty}\|^{*}_{i_{1},\dots, i_{m}} \leq C\, \|\overline{x-y}\|^{*}_{i_{1},\dots, i_{m}},
	\]
	for every $\{i_{1},\dots, i_{m}\} \subset \{1,\dots,n\}$.
	
	For any $\epsilon > 0$, choose $\delta = \frac{\epsilon}{C}$. Then, for $x,y \in X$ where
	$\|\overline{x-y}\|^{*}_{i_{1},\dots, i_{m}}< \delta$ for every $\{i_{1},\dots, i_{m}\} \in \{1,\dots,n\}$, we have
	\[
	\|\overline{Tx-Ty}\|^{*}_{i_{1},\dots, i_{m}} < \epsilon,
	\]
	for every $\{i_{1},\dots, i_{m}\} \in
	\{1,\dots,n\}$. Therefore, $T$ is continuous in $X$ 
\end{proof}

Note that, for an $m_{1} \in \{1,\dots,n\}$ a contractive mapping with respect to the norms
of class-$m_{1}$ collection are continuous with respect to the norms of class-$m$ collection,
for any $m \in \{1,\dots,n\}$. Moreover, we have the following theorem.

\begin{thm}\label{lm}
	Let $(X,\|\cdot,\dots,\cdot\|)$ be an $n$-normed space, $m \in \{1,\dots,n\}$ and $T : X \to X$.
	If $T$ is a contractive mapping with respect to the norms of class-$1$ collection, then $T$ is a
	contractive mapping with respect to the norms of class-$m$ collection.
\end{thm}

\begin{proof} Suppose that $T$ is a contractive mapping with respect to the norms of class-$1$
	collection, then there is a $C\in(0,1)$ such that for any $x,y \in X$ we have
	\begin{equation}\label{e1}
	\|\overline{Tx-Ty}\|^{*}_{j} \leq C \|\overline{x-y}\|^{*}_{j},
	\end{equation}
	for every $j \in \{1, \dots,n\}$. Let $m \in \{1,\dots,n\}$, by (\ref{e1}) we have
	\begin{eqnarray*}
		\|\overline{Tx-Ty}\|^{*}_{i_{1}} & \leq C \|\overline{x-y}\|^{*}_{i_{1},}\\
		& \vdots \\
		\|\overline{Tx-Ty}\|^{*}_{i_{m}} &\leq C \|\overline{x-y}\|^{*}_{i_{m}},
	\end{eqnarray*}
	for every $\{i_{1},\dots,i_{m}\} \in \{1,\dots,n\}$.
	It therefore follows from the above inequalities that
	\begin{eqnarray*}
		\|\overline{Tx-Ty}\|^{*}_{i_{1},\dots,i_{m}}
		& = \|\overline{Tx-Ty}\|^{*}_{i_{1}} + \dots + \|\overline{Tx-Ty}\|^{*}_{i_{m}} \\ & \leq C \left( \|\overline{x-y}\|^{*}_{i_{1}}
		+ \dots + \|\overline{x-y}\|^{*}_{i_{m}}\right)\\
		& = C \|\overline{x-y}\|^{*}_{i_{1},\dots,i_{m}},
	\end{eqnarray*}
	for every $\{i_{1},\dots,i_{m}\} \subset \{1,\dots,n\}$.
	If we consider the norms of class-$m$ collection, then this means  $T$ is a contractive mapping with
	respect to the norms of class-$m$ collection.
\end{proof}

\noindent We also have this following theorem which will be used later.

\begin{thm}\label{mn}
	Let $(X,\|\cdot,\dots,\cdot\|)$ be an $n$-normed space, $m \in \{1,\dots,n\}$ and $T : X \to X$.
	If $T$ is a contractive mapping with respect to the norms of class-$m$ collection, then $T$ is also a
	contractive mapping with respect to the norm of class-$n$ collection.
\end{thm}

\begin{proof} Let $T: X \to X$ be a contractive mapping with respect to the norms of
	class-$m$ collection for an  $m \in  \{1,\dots,n\}$. Then there is a $C \in (0,1)$ such that for
	every $x,y \in X$ we have
	\begin{equation}\label{eq8}
	\|\overline{Tx-Ty}\|^{*}_{i_{1},\dots,i_{m}} \leq C \|\overline{x-y}\|^{*}_{i_{1},\dots,i_{m}},
	\end{equation}
	for every $\{i_{1},\dots,i_{m}\} \subset \{1,\dots,n\}$. One can see that (\ref{eq8}) contains
	$\binom{n}{m}$ inequalities. From these inequalities, we have
	$$(n-1)\|\overline{Tx-Ty}\|^{*}_{1,\dots,n} \leq C (n-1) \|\overline{x-y}\|^{*}_{1,\dots,n},$$
	or
	$$\|\overline{Tx-Ty}\|^{*}_{1,\dots,n} \leq C \|\overline{x-y}\|^{*}_{1,\dots,n},$$
	which means that $T$ is a contractive mapping with respect to the norm of class-$n$ collection.
\end{proof}

Finally, we provide a fixed point theorem for a contractive mapping in a closed and bounded set on
an $n$-normed space, with respect to the norms of class-$m$ collection.

\begin{thm}\label{FPt}
	Let $(X,\|\cdot,\dots,\cdot\|)$ be an $n$-normed space, $m \in \{1,\dots,n\}$ and $K \subset X$
	is nonempty, closed and bounded.
	Suppose that $X$ is complete (with respect to the norms of class-$m$ collection). If T: $K \to K$
	is a contractive mapping with respect to the norms of class-$m$ collection, then $T$ has a unique fixed point.
\end{thm}

\begin{proof} Fix an  $m \in \{1,\dots,n\}$. Let $x_{0} \in K $ and $\{x_{k}\}$ be 	
	a sequence in $K$ such that
	\[
	x_{k}=T(x_{k-1})=T^{k}(x_{0}) \, \, \, \, ; \, \, \, k= 1, 2,\dots 
	\]
	Since $T$ is a contractive mapping, there is a $C \in (0,1)$
	such that for $x_{0},x_{1} \in K$ we have
	\begin{eqnarray*}
		\|\overline{T^{2}(x_{0})-T^{2}(x_{1})}\|^{*}_{i_{1},\dots,i_{m}} 	
		& =\|\overline{T(T(x_{0}))-T(T(x_{1}))}\|^{*}_{i_{1},\dots,i_{m}}\\
		& \leq  C \|\overline{T(x_{0})-T(x_{1})}\|^{*}_{i_{1},\dots,i_{m}}\\
		& \leq  C^{2} \|\overline{x_{0}-x_{1}}\|^{*}_{i_{1},\dots,i_{m}},
	\end{eqnarray*}
	
	\noindent for every $\{i_{1},\dots,i_{m}\} \subset \{1,\dots,n\}$. By using induction, we have
	\[
	\|\overline{T^{k}(x_{0})-T^{k}(x_{1})}\|^{*}_{i_{1},\dots,i_{m}} \leq  C^{k}
	\|\overline{x_{0}-x_{1}}\|^{*}_{i_{1},\dots,i_{m}},
	\]
	for every $\{i_{1},\dots,i_{m}\} \subset \{1,\dots,n\}$.
	Now we show that $\{x_{n}\}$ is a Cauchy sequence in $K$. Let $k,l \in \mathbb{N}$.
	Without loss of generality, we take $l>k$ and $l=k+p$, with $p \in \mathbb{N}$. Then we have
	\begin{eqnarray*}
		\|\overline{x_{k}-x_{l}}\|^{*}_{i_{1},\dots,i_{m}}	
		& = &\|\overline{x_{k}-x_{k+p}}\|^{*}_{i_{1},\dots,i_{m}}\\
		& \leq &\|\overline{x_{k}-x_{k+1}}\|^{*}_{i_{1},\dots,i_{m}} + \dots +
		\|\overline{x_{k+p-1}-x_{k+p}}\|^{*}_{i_{1},\dots,i_{m}}\\
		&=&\|\overline{T^{k}(x_{0})-T^{k}(x_{1})}\|^{*}_{i_{1},\dots,i_{m}}+ \dots +\\
		& & \|\overline{T^{k+p-1}(x_{0})-T^{k+p-1}(x_{1})}\|^{*}_{i_{1},\dots,i_{m}}\\
		& \leq & \left(C^{k} + \dots + C^{k+p-1}\right) \|\overline{x_{0}-x_{1}}\|^{*}_{i_{1},\dots,i_{m}},
	\end{eqnarray*}
	for every $\{i_{1},\dots,i_{m}\} \subset \{1,\dots,n\}$. Also, since $K$ is bounded, for any
	$x_{0},x_{1} \in K$ there exists an $M>0$ such that $\|\overline{x_{0}-x_{1}}\|^{*}_{i_{1},\dots,i_{m}} \leq M$,
	for every $\{i_{1},\dots,i_{m}\} \subset \{1,\dots,n\}$. Then we have
	\begin{eqnarray*}
		\|\overline{x_{k}-x_{l}}\|^{*}_{i_{1},\dots,i_{m}}  & \leq &\left(C^{k} + \dots + C^{k+p-1}\right) M \\
		& = & \left(C^{k} + \dots + C^{l-1}\right) M,
	\end{eqnarray*}
	for every $\{i_{1},\dots,i_{m}\} \subset \{1,\dots,n\}$. Since $C \in (0,1)$, we have
	\[
	\lim_{k,l \to \infty} \|\overline{x_{k}-x_{l}}\|^{*}_{i_{1},\dots,i_{m}}=0,
	\]
	for every  $\{i_{1},\dots,i_{m}\} \subset \{1,\dots,n\}$, which means that $\{x_{k}\}$ is a Cauchy sequence.
	
	Moreover, since $X$ is complete with respect to the norms of class-$m$ collection and $K$ is closed then $x_{k} \to x$,
	with $x \in K$. Propositions \ref{p1} and \ref{p2} imply that  
	\[
	T(x)=\lim_{k \to \infty} T(x_{k})=\lim_{k \to \infty} x_{k+1} = x.
	\]
	Therefore, $T$ has a fixed point in $K$ with respect to the norms of class-$m$ collection.
	Next, we want to show the uniqueness of the fixed point with respect to the norms of class-$m$ collection.
	Assume that $x^{\prime} \in K$ is another fixed point of $T$. Because $T$ is a contractive
	mapping, there is a $C \in (0,1)$ such that
	\begin{eqnarray*}
		\|\overline{x-x^{\prime}}\|^{*}_{i_{1},\dots,i_{m}} &= \|\overline{T(x)-T(x^{\prime})}\|^{*}_{i_{1},\dots,i_{m}}\\
		& \leq C \|\overline{x-x^{\prime}}\|^{*}_{i_{1},\dots,i_{m}},
	\end{eqnarray*}
	for every  $\{i_{1},\dots,i_{m}\} \subset \{1,\dots,n\}$. This is true only for $\|\overline{x-x^{\prime}}
	\|^{*}_{i_{1},\dots,i_{m}}=0$, for every  $\{i_{1},\dots,i_{m}\} \subset \{1,\dots,n\}$. This means that
	$x=x^{\prime}$ or $T$ has a unique fixed point. Therefore, for any $m \in \{1,\dots,n\}$, a contractive
	mapping $T$ has a unique fixed point.
\end{proof}

\section{Concluding Remarks}

Let us consider the $p$-summable sequences $\ell^{p}$ (for $1\le p\le \infty$) containing all sequences
of real numbers $x=(x_{j})$ for which $\sum |x_{j}|^{p} < \infty$.
As in \cite{gun2}, one may equip this space with the following $n$-norm
\begin{equation}
\|x_{1},\dots,x_{n}\|_{p}:=\left[\frac{1}{n!} \sum_{j_{1}} \cdots \sum_{j_{n}} \left| \det
\left[\xi_{ij_{k}}\right]_{i,k} \right|^{p} \right]^{\frac{1}{p}},
\end{equation}
with $x_{i}=\left(\xi_{ij}\right) \in \ell^{p}, i = 1, \dots, n.$

Now, one can see that if $\{y_{1},\dots,y_{n}\}$ is a linearly independent set
in $\ell^{p}$, the norm of class-$n$ collection $\|\cdot\|^{*}_{1,\dots,n}$ can be writen as
\begin{equation}\label{normstar}
\|x\|^{*}_{p}:= \left[\sum _{\{i_{2},\dots,i_{m}\} \subseteq \{1, \dots,n\}}\|x,y_{i_{2}},\dots,
y_{i_{n}}\|_{p}^{p} \right]^{\frac{1}{p}}.
\end{equation}
This is precisely the norm that was used in \cite{eka} as a bridge to prove the fixed
point theorem on $(\ell^{p},\|\cdot, \dots, \cdot\|_p)$.

In this paper, we provide more alternatives to study $\ell^{p}$ as an $n$-normed space. Note that
the usual norm on $\ell^{p}$ is defined by
\begin{equation}
\|x\|_{p}=\left[\sum_{i=1}^{\infty} \left|x_{i}\right|^{p}\right]^{\frac{1}{p}}.
\end{equation}
Furthermore, we have some equivalence relations between norm $\|\cdot\|_{p}$ dan $\|\cdot\|_{p}^{*}$.

\begin{thm}\label{equivalent}{\rm \cite{eka}}
	Let $\{y_{1},\dots,y_{n}\}$ be a linearly independen set on $\ell^{p}$. Then
	the norm $\|\cdot\|_{p}^{*}$ defined in (\ref{normstar}) is equivalent to the usual norm
	$\|\cdot\|_{p}$ on $\ell^{p}$. Precisely, we have
	\begin{multline*}
	\frac{n \|y_{1},\dots,y_{n}\|_{p}}{(2n-1)\left[\|y_{1}\|_{p}+\dots+\|y_{n}\|_{p}\right]}\|x\|_{p} \leq
	\|x\|_{p}^{*} \\
	\leq  (n!)^{1-\frac{1}{p}}\left[\sum_{\{i_{2},\dots,i_{n}\}\subset \{1,\dots,n\}} \|y_{i_{2}}\|_{p}^{p}
	\dots \|y_{i_{n}}\|_{p}^{p} \right]^{\frac{1}{p}}\|x\|_{p},
	\end{multline*}
	for every $x \in \ell^{p}$.
\end{thm}

Since $(\ell^{p},\|\cdot\|_p)$ is a Banach space, by Theorem \ref{equivalent} we have the following corollary.
\begin{cor}\label{cor}{\rm \cite{eka}}
	The normed space $(\ell^{p},\|\cdot\|_{p}^{*})$ is a Banach space.
\end{cor}

With our approach, we have the proposition below.

\begin{prop}\label{prop1}
	Let $\{y_{1},\dots,y_{n}\}$ be a linearly independent set on $\ell^{p}$. Then the norm of class
	$n$-collection which is defined in (\ref{e05}) is equivalent with the norm $\|\cdot\|_{p}^{*}$
	in (\ref{normstar}). Precisely we have
	$$\|\cdot\|_{p}^{*}\leq \|\cdot\|_{1,\dots,n}^{*}\leq (n)^{1-\frac{1}{p}} \|\cdot\|_{p}^{*}.$$
\end{prop}

\begin{proof} Consider the $n$-normed space $(\ell^{p},\|\cdot,\dots,\cdot\|_{p})$ and
	$\{y_{1},\dots,y_{n}\}$ be a linearly independent set on $\ell^{p}$. For any $x \in X$, recall the norm of class-$n$
	collection is
	\[
	\|x\|= \sum_{\{i_{2},\dots,i_{n}\}\subset \{1,\dots,n\}} \|x, y_{i_{2}},\dots,i_{n}\|_{p}^{*},
	\]
	and
	\[
	\|x\|_{p}^{*}= \left(\sum_{\{i_{2},\dots,i_{n}\}\subset \{1,\dots,n\}} \|x, y_{i_{2}},\dots,i_{n}\|_{p}^{p}\right)^{\frac{1}{p}}.
	\]
	Then we have
	\begin{eqnarray*}
		\|x\|_{p}^{*}&=& \left(\sum_{\{i_{2},\dots,i_{n}\}\subset \{1,\dots,n\}} \|x, y_{i_{2}},\dots,i_{n}\|_{p}^{p}\right)^{\frac{1}{p}}\\
		&\leq & \sum_{\{i_{2},\dots,i_{n}\}\subset \{1,\dots,n\}} \|x, y_{i_{2}},\dots,i_{n}\|_{p} \\
		& = & \|x\|_{1,\dots,n}^{*}.
	\end{eqnarray*}
	From Theorem (\ref{equivalent}) and by using H\"older  inequality we have
	\begin{eqnarray*}
		\|x\|_{1,\dots,n}^{*}&=& \sum_{\{i_{2},\dots,i_{n}\}\subset \{1,\dots,n\}} \|x, y_{i_{2}},\dots,i_{n}\|_{p}^{p}\\
		&\leq & (n)^{1-\frac{1}{p}} \left(\sum_{\{i_{2},\dots,i_{n}\}\subset \{1,\dots,n\}} \|x, y_{i_{2}},\dots,i_{n}\|_{p}^{p}\right)^{\frac{1}{p}}\\
		& = & (n)^{1-\frac{1}{p}} \, \, \|x\|_{p}^{*}.
	\end{eqnarray*}
	Then we have the inequality we want.
\end{proof}

Note that, we write $\|x\|_{1,\dots,n}^{*}$ instead of $\|\overline{x}\|_{1,\dots,n}^{*}$,
because the element of class-$n$ collection of an $n$-normed space is the $n$-normed space itself.
One might notice that the formula of the norm of class-$n$ collection is simpler than the norm
$\|\cdot\|_{p}^{*}$ that Ekariani \textit{et al} used and both norms are equivalent. Since
the norm $\|\cdot\|_{p}^{*}$ equivalent to the usual norm $\|\cdot\|_{p}^{*}$ on $\ell^{p}$,
we have the following corollaries.

\begin{cor}
	Let $\{y_{1},\dots,y_{n}\}$ be a linearly independent set on $\ell^{p}$. Then the norm
	$\|\cdot\|_{1,\dots,n}^{*}$ is equivalent to the usual norm $\|\cdot\|_{p}$ on $\ell^{p}$.
	Precisely we have
	\begin{multline*}
	\frac{n \|y_{1},\dots,y_{n}\|_{p}}{(2n-1)\left[\|y_{1}\|_{p}+\dots+\|y_{n}\|_{p}\right]}\|x\|_{p} \leq
	\|x\|_{p}^{*} \\
	\leq  (n \cdot n!)^{1-\frac{1}{p}}\left[\sum_{\{i_{2},\dots,i_{n}\}\subset \{1,\dots,n\}} \|y_{i_{2}}\|_{p}^{p}
	\dots \|y_{i_{n}}\|_{p}^{p} \right]^{\frac{1}{p}}\|x\|_{p},
	\end{multline*}
\end{cor}

\begin{cor}\label{cor-n}
	The normed space $(\ell^{p},\|\cdot\|_{1,\dots,n}^{*})$ is a Banach space.
\end{cor}

\noindent Now, we present a more general fixed point theorem on $(\ell^{p},\|\cdot,\dots,\cdot\|_p)$ as follows.

\begin{thm}\label{fpl}
	Consider the $n$-normed space $(\ell^{p},\|\cdot,\dots,\cdot\|_{p})$ and let $K\subseteq \ell^{p}$ be nonempty,
	closed and bounded. If $T:K\to K$ is a contractive mapping with respect to norms of class-$m$ collection
	for an $m \in \{1,\dots,n\}$, then $T$ has a unique fixed point.
\end{thm}

\begin{proof} Fix an $m \in \{1,\dots,n\}$ and let $T:K\to K$ be a contractive mapping with
	respect to norms of class-$m$ collection on $\ell^{p}$. By Theorem \ref{mn}, $T$ is contractive with
	respect to the norm of class-$n$ collection on $\ell^{p}$. By Corollary (\ref{cor-n}) $\ell^{p}$ is complete
	with respect to the norm of class-$n$ collection. Hence, $T$ has a unique fixed point. 
\end{proof}

\begin{rem}
	We can also prove Theorem \ref{fpl} using Theorem \ref{FPt}, noting the fact that $\ell^{p}$ is
	complete with respect to the norms of class-$m$ collection (see Theorem \ref{comp}).
\end{rem}

\subsection*{Acknowledgment}
This work is supported by ITB Research and Innovation Program 2019. 
The first author is also supported by LPDP Indonesia.

\end{document}